\documentclass[11pt,reqno]{amsart}

\usepackage{color}
\usepackage{amsmath, amsthm, amssymb}
\usepackage{amsfonts}
\usepackage[ansinew]{inputenc}
\usepackage[dvips]{epsfig}
\usepackage{graphicx}
\usepackage[english]{babel}
\usepackage{hyperref}

\usepackage{cite}
\usepackage{graphicx}
\usepackage{amscd}
\usepackage{color}
\usepackage{bm}
\usepackage{enumerate}

\usepackage{verbatim}
\usepackage{hyperref}
\usepackage{amstext}
\usepackage{latexsym}

\theoremstyle{plain}
\newtheorem{theorem}{Theorem}[section]

\newtheorem{corollary}[theorem]{Corollary}

\newtheorem{remark}[theorem]{Remark}

\numberwithin{theorem}{section}
\numberwithin{equation}{section}

\newcommand{\average}{{\mathchoice {\kern1ex\vcenter{\hrule height.4pt
width 6pt depth0pt} \kern-9.7pt} {\kern1ex\vcenter{\hrule
height.4pt width 4.3pt depth0pt} \kern-7pt} {} {} }}

\def\R{\mathbb{R}}

\renewcommand{\a }{\alpha }
\renewcommand{\b }{\beta }
\renewcommand{\d}{\delta }

\newcommand{\D }{\Delta }

\newcommand{\e }{\varepsilon }
\newcommand{\g }{\gamma}

\newcommand{\n }{\nabla }
\newcommand{\vp }{\varphi }

\newcommand{\s }{\sigma }

\renewcommand{\t }{\tau }

\renewcommand{\O }{\Omega }

\newcommand{\ov}{\overline}

\newcommand{\be}{\begin{equation}}
\newcommand{\ee}{\end{equation}}

\newcommand{\de}{\partial}

 \usepackage{mathrsfs}

\newcommand{\N}{\mathbb{N}}

\newcommand{\dist}{{\rm dist}}

\newcommand{\eps}{\varepsilon}

\renewcommand{\epsilon}{\varepsilon}

\newcommand{\Ds}{ (-\D)^s}

\begin{document}
\title{Gradient estimates in fractional Dirichlet problems}
 
\author[Mouhamed Moustapha Fall]
{Mouhamed Moustapha Fall}
\address{African Institute for Mathematical Sciences (AIMS) Senegal. KM2 Route de Joal Mbour. BP 1418, Senegal}
\email{mouhamed.m.fall@aims-senegal.org}

\author[Sven Jarohs]
{Sven Jarohs}
\address{Institut f\"ur Mathematik,
Goethe-Universit\"at Frankfurt.
Robert-Mayer-Str. 10
D-60629 Frankfurt am Main, Germany}
\email{jarohs@math.uni-frankfurt.de}

\thanks{M.M. Fall's work is supported by the Alexander
von Humboldt foundation. Part of this work was done while M.M. Fall was visiting
the Goethe University in Frankfurt am Main during   July-August 2019 and he thanks the Mathematics department for their kind hospitality. The authors are grateful to  Tobias Weth and Xavier Ros-Oton   for    many useful discussions.  We also thank Xavier Ros-Oton for handing us \cite{AR19}.  }


 \begin{abstract}
   \noindent
We obtain some fine gradient estimates near the boundary for solutions to fractional elliptic problems subject to exterior Dirichlet boundary conditions. Our results provide, in particular, the sign of the normal derivative of such solutions near the boundary of the underlying domain.
 \end{abstract}

\maketitle

\section{Introduction}\label{s:Int}
Let $\O$  be an open bounded subset of $\R^N$ with $C^{1,1}$ boundary and let $s\in (0,1)$. In this paper we analyze the boundary behavior of distributional solutions $u\in C^s(\R^N) $   to the equation
\be\label{eq:main}
\Ds u=f(x,u) \quad\textrm{ in $\O$},  \qquad u=0 \quad\textrm{ in $\R^N\setminus \O$},
\ee
where $f\in L^\infty_{loc}(\R^N\times \R)$ and  $(-\Delta)^s$ denotes the fractional Laplacian and is defined for $\vp\in C^\infty_c(\R^N)$ by  
\[
(-\Delta)^s \vp (x):=c_{N,s}\lim_{\eps\to 0^+}\int_{\R^N\setminus B_{\eps}(0)}\frac{\vp (x)-\vp (x+y)}{|y|^{N+2s}}\ dy,\quad x\in \R^N
\]
with a normalization constant $c_{N,s}=\frac{s4^s\Gamma(\frac{N}{2}+s)}{\pi^{N/2}\Gamma(1-s)}$. 
Here and in the following, we assume,    in the case $s\in (0,1/2]$ that, for some   $\s\in ( 1-2s,1)$ and for all $M>0$, there exists a constant $A_M$ such that 
\be\label{eq:Assump-entries}
\sup_{t\in (-M,M)}[f(\cdot,t)]_{C^{0,\s}(\ov\O)}+  \sup_{x\in \O }[f(x,\cdot )]_{C^{0,1}[-M,M]}\leq A_M.
\ee
We note that, under the assumptions on $f$, the solution $u$ to \eqref{eq:main} belongs to $ C^{1}_{loc}(\O)$ by the interior regularity theory, see e.g. \cite{RS16b}.

\medskip

To study the boundary behavior of $u$, we consider a function  $\d$, which coincides with the distance function $ \dist(\cdot,\R^{N}\setminus \Omega)$ in a neighborhood of $\de\O$ and in $\R^N\setminus\O$. Moreover, we suppose that $\d$ is positive in $\O$ and  $\d \in C^{1,1}(\ov\O)$.

\medskip

Letting  $\psi={u}/{\d^s}$,  the known boundary regularity theory for fractional elliptic equations  (see e.g.  Ros-Oton and Serra \cite{RS-2} followed by \cite{RS16a,RS16b,Fall-reg-1, RS-1})  states that, for any  $\a\in(0,s)$,
\be \label{eq:estim-psi-intro}
\|\psi\|_{C^\a(\ov \O)}\leq C \sup_{x\in\O } |f(x,u(x))|,
\ee
where $C=C(N,s,\a,\O)$ is a positive constant.   
Moreover,  since $u= \d^s\psi $, we have that
\be \label{eq:Class-estim}
\d^{1-s}(x) \n u (x)=s\psi(x)\n\d(x)+\d(x)\n \psi(x)  \qquad\textrm{   for  all $x\in  \O$.} 
\ee
However, the identity \eqref{eq:Class-estim} and the estimate in  \eqref{eq:estim-psi-intro}   do not provide a fine asymptotic  of $  \n  u(x)$ near $\de\O$, since one cannot deduce from  \eqref{eq:estim-psi-intro}  a pointwise estimate of $\n \psi$ near $\de\O$. In particular, the monotonicity of $u$ in the normal direction near the boundary is in general not known and cannot be deduced from the fractional Hopf lemma, which provides only the sign of $\psi$ on $\de\O$, see e.g. \cite{FJ15}. The purpose of the present paper is to investigate these questions and we show  that, for some $\b>0$,
\be\label{eq:to-get} 
\d^{1-s} \n u\in {C^\b(\ov\O)}  \qquad \textrm{and}  \qquad \d^{1-s}(x) \n u(x)\cdot\n \d(x)=s\psi(x)  \qquad\textrm{   for all $x\in \de \O$.} 
\ee

We emphasize that under the assumptions on $f$,   \eqref{eq:to-get}  does not  follow  from the known boundary regularity theory for fractional elliptic equations even if $\O$ is of class $C^\infty$. Indeed, by  the  results of Grubb  \cite{Grubb1}, we have that $ \psi \in C^{\a}(\ov\O)$ for all $\alpha\in (0,s)$ and also if $2s\leq 1$ then by \eqref{eq:Assump-entries},      $\psi\in C^{s+\min (s,\s)}(\ov\O)$, provided $s+\min (s,\s)\not\in \N$. Clearly,  each of these H\"older regularity on $\psi$   does not  imply a pointwise  estimate of $\n \psi$ and cannot imply \eqref{eq:to-get}.

Our first main result is the following.
\begin{theorem}\label{th:Gradient-estimate-2s-geq1}
Let $N\geq 1$, $s\in (1/2,1)$, $\a\in (0,1)$ and  $\O\subset \R^N$ be an open bounded set of class $C^{1,\g}$, with $\g>s$. Let $u \in C^s(\R^N) $ and $g\in L^\infty(\R^N)$  be such that 
\be\label{eq:main-2s-geq1}
\Ds u=g \quad\textrm{ in $\O$},  \qquad u=0 \quad\textrm{ in $\R^N\setminus \O$}.
\ee
Let $\psi={u}/{\d^s}$ satisfy 
\be \label{eq:estim-psi-intro-2sgeq1}
\|\psi\|_{C^\a(\ov \O)}\leq C_0 \left( \|g\|_{L^\infty(\R^N)}+ \|u\|_{L^\infty(\O)}\right),
\ee
for some constant $C_0>0$.
  Then  provided $\a\not=s$, we have
\be  \label{eq:est-1-psi-intro}
|\n \psi(x)|\leq C  \d^{\min(\a,s)-1}(x)\left( \|g\|_{L^\infty(\R^N)}+ \|u\|_{L^\infty(\O)}\right) \qquad\textrm{ for almost all $x\in \O$.}
\ee
  If moreover $\O$ is of class $C^{1,1}$,  then    for all $\b\in (0,\min(\a,2s-1))$,  
\be \label{eq:cont-weight-grad-intro}
\|\d^{1-s} \n u\|_{   {C^\b(\ov\O)}    }\leq C   \left( \|g\|_{L^\infty(\R^N)}+ \|u\|_{L^\infty(\O)}\right)
\ee
and 
\be\label{eq:equlaity-norm-deriv}
\d^{1-s}(x) \n u(x)\cdot\n \d(x)=s\psi(x)  \qquad\textrm{   for all $x\in \de \O$.}  
\ee
Here,    $C=C(\O,N,s,\a,\g, \b,C_0)$.
\end{theorem} 

 We recall that by \cite{RS16b,AR19},  if $\O$ is  of class $C^{2,\e}$  and $g\in C^\e(\R^N)$, for some $\e>0$,  then  \eqref{eq:estim-psi-intro-2sgeq1} holds for some $\a>s$. In this case,  $\d^{\min(\a,s)-1}(x)$ in \eqref{eq:est-1-psi-intro} can be replaced with $\d^{s-1}(x)$. \\

To state our next results, we will consider a function $U\in C^s(\R^N)\cap C^1_{loc}(\O)$ satisfying
\be\label{eq:torsion}
\Ds U\in C^\s(\ov\O) 
\ee
and 
\be\label{eq:estimU}
  c \d^s(x)\leq U(x) \leq \frac{1}{c} \d^s(x) \qquad\textrm{ for all $x\in \R^N$},
\ee
for some positive constant $c$.  

\medskip

Our second main result is the following.
\begin{theorem}\label{th:Gradient-estimate-gen}
Let $N\geq 1$, $s\in (0,1/2]$, and  $\O\subset \R^N$ be an open bounded set  of class $C^{1,1}$. Let $u \in C^s(\R^N) $ be a solution to \eqref{eq:main}, where $f$ satisfies \eqref{eq:Assump-entries} and let $U$ satisfy \eqref{eq:torsion} and \eqref{eq:estimU}.    Suppose that  $u$  satisfies    \eqref{eq:estim-psi-intro}, for some $\a\in (0,1)$ with   $\a\not=s$.
Let $\Psi:=\frac{U}{\d^s}$ and suppose that $\Psi\in {C^\a(\ov \O)} $.  Then   the following statements holds.
\begin{enumerate}
\item[(i)] We have  
\be \label{eq:grad-estim-gen}
|\n  \psi(x)|\leq   C\left( \d^{\min(s,\a )-1}(x)+  |\n  \Psi(x)| \right)  \qquad\textrm{ for all $x\in \O$.}
\ee
\item[(ii)]  If    $ \d^{1-s} \n U\in C^{\g}(\ov \O) $, for some $\g>0$, then     for all $\b\in (0, \min\{\g,\a ,s , \s-1+2s \}],$  
\be \label{eq:normal-Hold-intro}
\|\d^{1-s} \n u\|_{   {C^{\b}(\ov\O)}    }\leq C  .
\ee
\end{enumerate}
For $M:=\|u\|_{L^\infty(\R^N)}$,   the constant $C$ above depends only on   $N$, $s$, $\b$, $\O$,  $\g$,  $\s$, $\a$, $A_M$, $U$ and $\|f\|_{L^\infty({\O}\times[-M,M])}$.

\end{theorem} 

As an example of a function $U\in C^s(\R^N)\cap C^{1}_{loc}(\O)$ satisfying  \eqref{eq:torsion} and \eqref{eq:estimU}, we can consider   the solution to 
\be\label{eq:torsion-0}
 \Ds U=1\qquad\textrm{in $\O$} \qquad\textrm{ and }\qquad U=0  \qquad\textrm{in $\R^N\setminus \O$.}
\ee
Here,  by  \cite{Grubb1}, if $\O$ is of class $C^\infty$ then $\Psi=U/\d^s\in C^\infty(\ov\O)$.    Therefore, by combining \eqref{eq:Class-estim},   \eqref{eq:normal-Hold-intro} and  \eqref{eq:grad-estim-gen}, we get     \eqref{eq:to-get}.

\begin{remark}\label{rem:Ok-aniso}
We notice that the results in Theorem \ref{th:Gradient-estimate-2s-geq1} and  Theorem \ref{th:Gradient-estimate-gen} remain valid if we replace $\Ds$ with the \textit{anisotropic fractional Laplacian} $\Ds_a$ with $a\in L^\infty(S^{N-1})$ if $2s\leq 1$ and $a\in C^\s(S^{N-1})$  for some $\s>1-2s$ if $2s\leq 1$. Here, the anisotropic fractional Laplacian is defined for   $\vp \in C^\infty_c(\R^N)$ as
\[
\Ds_a\vp (x):=\lim_{\eps\to0^+}\int_{\R^N\setminus B_\epsilon(0)}\Big(\vp (x)-\vp (x+y)\Big)\frac{a(y/|y|)}{|y|^{N+2s}}\ dy,\quad x\in \R^N.
\] 
In these cases, by \cite{RS16a},  the interior and boundary regularity  that is needed in the proofs in Section \ref{eq:Proofs} below remains valid.
\end{remark}
Our next result is a consequence of Theorem \ref{th:Gradient-estimate-gen} and the recent results in \cite{AR19} where the authors show the existence of a function satisfying  \eqref{eq:torsion} and \eqref{eq:estimU} in $C^{1,1}$ domains.
\begin{corollary} \label{cor:sleq12}
Let $N\geq 1$, $s\in (0,1/2]$, and  $\O\subset \R^N$ be an open bounded set  of class $C^{1,1}$. Let $u \in C^s(\R^N) $ be a solution to \eqref{eq:main} satisfying  \eqref{eq:estim-psi-intro}, where $f$  satisfies \eqref{eq:Assump-entries}. Provided $\a\not=s$, the following statements hold.
\begin{enumerate}
\item[(i)] We have  
\be\label{eq:estim-Grad-ok}
|\n  \psi(x)|\leq   C \d^{\min(s,\a )-1}(x)    \qquad\textrm{ for all $x\in \O$.}
\ee
\item[(ii)] For all $\b\in (0, \min\{\a ,s , \s-1+2s \}],$  
\be \label{eq:normal-Hold-intro2}
\|\d^{1-s} \n u\|_{   {C^{\b}(\ov\O)}    }\leq C  
\ee
and 
\be\label{eq:norm-deriv-sleq12}
\d^{1-s}(x) \n u(x)\cdot\n\d(x)=s\psi(x) \qquad\textrm{ for all  $x\in \de\O$.} 
\ee
\end{enumerate}
Here, for $M:=\|u\|_{L^\infty(\R^N)}$,   the constant $C$ above depends only on   $N$, $s$, $\b$, $\O$,    $\s$, $\a$, $A_M$  and $\|f\|_{L^\infty({\O}\times[-M,M])}$.
\end{corollary}

\begin{remark}  In view of the above results, the following question remain open.
Our  arguments  yield a bound of $|\n\psi|$ in terms of $\d^{\min(s,\a )}$. Does the estimate $ |\n \psi |\leq C \d^{\a-1}$  hold for some $\a>s$?
\end{remark}

To prove Theorem \ref{th:Gradient-estimate-2s-geq1}, we  consider the function  $y\mapsto v_x(y):= \d^s (y) (\psi (y)-\psi(x))$, with $y\in B(x,\d(x))$. Note  that $v_x(y)=u(y)-\d^s(y)\psi(x)$ and its order of vanishing near $\de\O$ is   $\d^{s+\min(\a,s)}(x)$.  We then apply interior regularity theory to the translated and rescaled equation for $v_x$ to deduce the estimate $\|v_x\|_{C^{1,\b}( B(x,\d(x)/2) )}\leq C\d^{s+\min(\a,s)}(x) $, from which we conclude the proof. 
 In the case of Theorem \ref{th:Gradient-estimate-gen}, we adopt the same strategy as in the proof of Theorem \ref{th:Gradient-estimate-2s-geq1}. However,   since we do not know a sharp result for the H\"older continuity of    $\Ds\d^s$ in $C^{1,1}$ domains, we replace $\d^s$ with $U$ in the definition of $v_x$.  We then apply interior regularity theory and a bootstrap argument  to the translated and rescaled problem.    

\section{Proof of  the main results }\label{eq:Proofs}
 We recall the interior regularity for the fractional Laplacian for equation to $\Ds v= g$ in $B_1$ with $v\in L^\infty(\R^N)$. Then, see e.g. \cite{RS16a}, we have the following estimates with a constant $C$ depending only on $N,s$ and $\t$.
\begin{enumerate}
\item  If $2s> 1$ and $\t\in (0,2s-1)$,
\be\label{eq:Intestim2}
\|v\|_{C^{1,\t}(B_{1/2})}\leq C ( \|g\|_{L^{\infty}(B_1)}    + \|v\|_{L^\infty(B_1)} +\|v\|_{L^1_s}) .
\ee
\item  If   $2s+\t\not\in \N$,   then
\be \label{eq:Intestim1}
\|v\|_{C^{2s+\t}(B_{1/2})}\leq C ( \|g\|_{C^{\t}(B_1)}  + \|v\|_{L^\infty(B_1)}+\|v\|_{L^1_s}) .
\ee
\end{enumerate} 
Here and in the following $B_t:=B(0,t)$ denotes the centered ball of radius $t>0$ in $\R^N$  and 
\[
\|v\|_{L^1_s}=\int_{\R^N}\frac{|v(x)|}{1+|x|^{N+2s}}\ dx.
\] 

\subsection{Proof of  Theorem \ref{th:Gradient-estimate-2s-geq1}} \label{ss:Th1}

For simplicity, we will assume that $$\|u\|_{L^\infty(\R^N)}+ \|g\|_{L^\infty(\R^N)}\leq 1.$$
From now on, $C$ always denotes a positive constant depending on $N$, $\O$, $s$, $\s $, $\a$, $\b$ and $\g$, which may change from line to line. \\

Since $\O$ is of class $C^{1,\g}$, we can assume that $\d$, defined Section \ref{s:Int}, is Lipschitz continuous in $\R^N$.   Fix $x\in \O $ and for $z\in \R^N$ we define
$$ 
u_x(z):=u(x+z\d(x)), \quad  \d_x(z):=\d(x+z\d(x)),\quad  \text{and}\quad v_x(z):=u_x(z)-\d_x^s(z)\psi(x).
$$

Since $\d$ is Lipschitz continuous,   for $z\in B_{1/2}$ we have that 

\be\label{estm:distance-function}
\frac{1}{2}\d(x)\leq \d_x(z)\leq 2\d(x) \quad\text{and}\quad |\n \d_x(z)|\leq C\d(x).
\ee
Since $\psi\in C^{\a }(\overline{\O })$ and $\d^s\in C^s(\R^N)$, we have for $z\in \R^N$
\[
|v_x(z)|=\d_x^s(z)|\psi(x)-\psi(x+\delta(x)z)|\leq C\d^{s+\a}(x)|z|^{\a}(1+|z|^s)1_{\frac{\O-x}{\d(x)}}(z),
\]
where we used that $\d$ is zero on $\R^N\setminus \O $.  As a consequence, for $z\in B_1$ we have
\be\label{estm-v1}
|v_x(z)|=\d_x^s(z)|\psi(x)-\psi(x+\delta(x)z)|\leq   C\d^{s+\a }(x).
\ee
We observe that for some $R=R(\O)$, we have $\frac{\O-x}{\d(x)}\subset B_{R/\d(x)}$.
 In particular, if $\a \neq s$,  
\begin{align*}
\|v_x\|_{L^1_s}&\leq C\d^{s+\a }(x)+C\d^s(x)\int_{B_{R/\d(x)}\setminus B_1} \frac{ \d^{\a }(x)|z|^{\a }}{|z|^{N+s}}\ dz\notag\\
&\leq C\d^{s+\a }(x)+C\d^{s+\a }(x)\int_1^{R/\d(x)}  t^{\a-1-s}\ dt\notag\\
&\leq C\d^{s+\a }(x)(1+\d^{s-\a }(x)).
\end{align*} 
We then  conclude that, for $\a\not=s$, 
\be  \label{estm-v2}
\|v_x\|_{L^1_s}\leq C \d^{s+\min(s,\a )}(x).
\ee
Next, we note that by the scaling properties of the fractional Laplacian, we have for $z\in B_1$

\be\label{eq:ds-applied}
\Ds v_x(z)=\delta^{2s}(x) g(x+\delta(x)z) -\psi(x)\delta^{2s}(x)[\Ds\delta^s](x+\delta(x)z).
\ee
We now complete the proof of the theorem.
\begin{proof}[Proof of Theorem \ref{th:Gradient-estimate-2s-geq1} completed]
We start by recalling that by \cite[Proposition 2.6]{RS-1}, if $\O$ is of class $C^{1,\g}$ with  $\g>s$,  then 
$$
|\Ds\d^s(x)|\leq  C \qquad\textrm{ for all $x\in \O$.}
$$
From the assumptions on $g$,   \eqref{estm-v1}, and \eqref{estm-v2} the estimate \eqref{eq:Intestim2} applied to the equation \eqref{eq:ds-applied} gives, for $\b\in (0, 2s-1)$ and $\a\not=s$,
\begin{align} \label{eq:from-interior-reg1}
\|v_x\|_{C^{1,\b}(B_{1/4})}&\leq C\left( \d^{2s}(x)+\d^{2s}(x)\sup_{z\in B_{1/2}}|\Ds\d^s(x+\d(x)z)| +\|v_x\|_{L^{\infty}(B_{1/2})}+\|v_x\|_{L^1_s} \right) \nonumber\\
&\leq C\left(\d^{2s}(x)  +\d^{s+\min(s,\a )}(x) \right)\leq C\d^{s+\min(s,\a )}(x).
\end{align}
We then deduce from this that, for all $x\in \O$
\be \label{eq:zzz1}
|\n u(x)-\psi(x)\n \d^s(x)|=\d^{-1}(x)|\nabla v_x(0)|\leq C\d^{s+\min(s,\a )-1}(x).
\ee
 Since $\d^s(x)\n \psi (x)= \n u(x)-\psi(x)\n \d^s(x)$, we get  \eqref{eq:est-1-psi-intro}.\\
 
 We now prove \eqref{eq:cont-weight-grad-intro} and we recall our assumption that $\O$ is of class $C^{1,1}$. By  \eqref{eq:from-interior-reg1}, for $\a\not=s$, we get
\be\label{eq:Z1}
\|\n u-\psi (x)\n \d^s \| _{L^\infty(B(x,\d(x)/4))}\leq C  \d^{s+\min(s,\a )-1}(x)
\ee
and 
\be \label{eq:Z2}
\begin{split}
[\n u -\psi(x)\n \d^s]_{C^{0,\b}( B(x,\d(x)/4) )}&=\d^{-1-\b}(x) [\n u -\psi(x)\n \d^s]_{C^{0,\b}( B(0,1/4) )}\\
& \leq C\d^{s+\min(s,\a )-1-\b}(x).	
\end{split}
\ee
Define $w_x(y):=\d^{1-s}(y) \left( \n u(y)-\psi (x)\n \d^s  (y)\right)$. Then, for $y_1,y_2\in  B(x,\d(x)/4) )$ and by \eqref{eq:Z1} and \eqref{eq:Z2}, we have 
\begin{align*}
|w_x(y_1)-&w_x(y_2)  |\\
&\leq  C  \left(\d^{-s}(x) |y_1-y_2| \d^{s+\min(s,\a )-1}(x) + \d^{1-s}(x) \d^{s+\min(s,\a )-1-\b}(x) |y_1-y_2|^\b\right)\\
&\leq  C  \left( \d^{1-\b -s}(x) |y_1-y_2|^{\b} \d^{s+\min(s,\a )-1}(x) +  \d^{ \min(s,\a ) -\b}(x) |y_1-y_2|^\b\right)\\
&\leq C \d^{ \min(s,\a )-\b}(x) |y_1-y_2|^{\b} ,
\end{align*} 
 where we used that $\d\in C^1(B(x,\d(x)/4) ) )$ and \eqref{estm:distance-function}.
 Hence form this and \eqref{eq:Z1}, provided $\b\in(0,\min (\a,s,2s-1))$, we get
$$
 [w_x]_{   C^\b(B(x,\d(x)/4) ) )} \leq C  \d^{ \min(s,\a )-\b}(x) \leq C .
$$
Therefore, noticing that $\d^{1-s}(y)\n u(y)=w_x(y)-\psi(x)\n \d(y)$,  $\n \d\in C^{0,1 }(B(x,\d(x)/4) ) $  and $|\psi(x)|\leq C$, we find, for all $\b\in(0,\min ( \a,s,2s-1))$, that 
$$
 [\d^{1-s}\n u ]_{   C^\b(B(x,\d(x)/4) ) )} \leq C.
$$
It then follows from a very similar argument as  in the proof of \cite[Proposition 1.1]{{RS-2}} that $[\d^{1-s}\n u ]_{C^\b(\ov\O)}\leq C$. Since $ \|\d^{1-s}\n u \|_{L^\infty(\O)}\leq C $ by \eqref{eq:zzz1}, we get \eqref{eq:cont-weight-grad-intro}.\\

Finally, since  $\O$ is of class $C^{1,1}$, then  $\n\d\in C(\ov\O)$ and $\n\psi\in C(\O)$. Therefore  \eqref{eq:equlaity-norm-deriv}    follows from \eqref{eq:Class-estim},  \eqref{eq:est-1-psi-intro} and \eqref{eq:cont-weight-grad-intro}. The proof      is thus  complete.
\end{proof}

\subsection{Proof of Theorem \ref{th:Gradient-estimate-gen}}
To simplify the write up, we assume for the following that with $M=\|u\|_{ L^\infty(\O)}$ we have $\|f\|_{L^{\infty}({\O} \times [-M,M])}\leq 1$ and $A_M\leq 1$ (recall \eqref{eq:Assump-entries}).\\
From now on, the letter $C$ always denotes a positive constant depending on $N$, $\O$, $s$, $\s $, $\a$, $\b$, $\g$, $U$  and $M$, which may change from line to line. 

\medskip

By \eqref{eq:estimU} and \eqref{eq:estim-psi-intro}, we see that $\ov \psi=\frac{u}{U}=\frac{\psi}{\Psi} \in C^\a(\ov \O)$.
In the following, we fix $r_0\leq c^{\frac{1}{s}}$, with $c$ being the constant appearing in  \eqref{eq:estimU}.
Recalling \eqref{eq:estimU}, for  $x\in \O $ and  $z\in \R^N$,  we define
$$ 
\ov u_x(z):=u(x+zU^{\frac{1}{s} }(x)), \quad \ov  U_x(z):=U(x+zU ^{\frac{1}{s} }(x)),\quad  \text{and}\quad \ov v_x(z):=\ov u_x(z)-\ov U_x (z) \ov \psi(x).
$$
We observe that, since $\ov \psi=u/U$,  
$$
  \ov v_x(z)= \ov U_x (z)\left(  \ov\psi_x(z)- \ov \psi(x) \right).
$$
Therefore in view of \eqref{eq:estimU} and the fact that $\ov \psi\in C^\a(\ov \O)$,  by using similar argument as   in the beginning of Section \ref{ss:Th1}, we find that, provided $\a\not=s$, 
\be  \label{estm-v2-gen}
\|\ov v_x\|_{L^1_s}\leq C \d^{s+\min(s,\a )}(x)
\ee
and 
\be\label{estm-v1-gen} 
\|\ov v_x\|_{L^\infty( B_{r_0/2} )}\leq C \d^{s+\min(s,\a )}(x).
\ee
Now direct computations, based on the scaling property of the fractional Laplacian and \eqref{eq:torsion}, yield for all $r_0\leq c^{\frac{1}{s}}$ and  $z\in B_{r_0/2}$
\be\label{eq:ds-applied-gen}
\Ds \ov  v_x(z)=U^{2}(x) f\left(x+ U^{\frac{1}{s}} (x)z,\ov u_x(z)\right) - \ov \psi(x) U ^{2}(x)\Ds U(x+U^{\frac{1}{s}} (x) z ) .
\ee
We now complete the proof of the theorem.
\begin{proof}[Proof of Theorem \ref{th:Gradient-estimate-gen} completed]
We start by recalling that $2s\leq 1$.
Let 
\be\label{estm-gx0-gen}
\ov g_x(z)=f\left(x+U^{\frac{1}{s}} (x)z,\ov u_x(z) \right)=f\left(x+U^{\frac{1}{s}} (x)z, \ov v_x(z) +\ov \psi(x)\ov U_x(z) \right).
\ee
Then from the assumptions on $f$ and \eqref{eq:estimU}, we get 
\be\label{estm-gx-gen}
\|\ov g_x\|_{C^{\min(s,\s )}(B_{r_0/2})}\leq C\left( 1+ \|\ov v_x\|_{C^{s}(B_{r_0/2})}  \right).  
\ee
By \eqref{eq:torsion},  \eqref{eq:Intestim1} and  \eqref{eq:ds-applied-gen}, provided $\t_0=2s+\min(s,\s) \notin \N$,    we have
\begin{align*}
\|\ov v_x\|_{C^{\t_0}(B_{r_0/4})}&\leq  C\left( U^{2}(x) \|\ov g_x\|_{C^{\min(s,\s )}(B_{r_0/2})}+U^{2}(x) +\|\ov v_x\|_{L^{\infty}(B_{r_0/2})}+\|\ov v_x\|_{L^1_s} \right),
\end{align*}
and if $\t_0\in \N$, we can replace $\|\ov v_x\|_{C^{\t_0}(B_{r_0/4})}$ above with   $\|\ov v_x\|_{C^{\t_0-\e}(B_{r_0/4})}$   for an arbitrary small $\e>0$.
Hence using \eqref{estm-gx-gen},  \eqref{estm-v1-gen}, and \eqref{estm-v2-gen} we obtain
\begin{align}\label{eq:start-iterate}
\|\ov v_x\|_{C^{\t_0}(B_{r_0/4})} \leq C\left( \|\ov v_x\|_{C^{\min(s,\s )}(B_{r_0/2})} +1\right)\d^{s+\min(s,\a )}(x)\leq C \d^{s+\min(s,\a )}(x).
\end{align}
   We consider a sequence of numbers $r_i=c^{\frac{1}{s}} 2^{-i-2}$ and $\tau_{i+1}=\min(2s+\tau_i,\s)$ for $i\in \N$. Then by \eqref{eq:ds-applied-gen} and \eqref{estm-gx0-gen},  for all $z\in B_{r_{i}}$, $i\in \N$ we have
\[
\Ds \ov v_x(z)= U^{2}(x) \ov g_x(z) -\ov \psi(x) U ^{2}(x)\Ds U(x+U^{\frac{1}{s}} (x) z ) .
\]
Hence iterating the above argument, provided $\t_{i+1}\not\in \N$ (or else we replace $\t_{i+1}$ with  $\t_{i+1}-\e$ for an arbitrary small $\e>0$), we get
\[
\| \ov v_x\|_{C^{\t_{i+1}}(B_{r_{i+1}})}\leq C\left( \|\ov v_x\|_{C^{\t_i}(B_{r_i})} +1\right) \d^{s+\min(s,\a )}(x).
\]
Therefore by \eqref{eq:start-iterate} and the fact that $1>\s >1-2s$, we must have that  for some $i_0\in \N$,   
\be \label{eq:after-iterr}
\|\ov v_x\|_{C^{1,\s-1+2s}(B_{r_{i_0}})}\leq C \d^{s+\min(s,\a )}(x).
\ee
  This implies, in particular that
$$
|\n \ov v_x(0)|\leq C  \d^{s+\min(s,\a )}(x).\qquad\textrm{ for all $x\in \O$,}
$$
Therefore, since    by  \eqref{eq:estimU}, 
\[
|\n u(x)-\ov \psi(x) \n  U (x)|=U^{-\frac{1}{s}}(x)|\nabla \ov v_x(0)|\leq C\d^{s+\min(s,\a )-1}(x).
\]
Using that $|U \n \ov\psi|=|\n u(x)-\ov \psi(x) \n  U (x)|$ and \eqref{eq:estimU},  we then get
$$
|\n \ov \psi(x)|\leq   C\d^{\min(s,\a )-1}(x)  \qquad\textrm{ for all $x\in \O$}
$$
and thus,  since   $\ov \psi=\frac{\psi}{\Psi} $, 
$$
|\n  \psi(x)|\leq   C\left( \d^{\min(s,\a )-1}(x)+  |\n  \Psi(x)| \right)  \qquad\textrm{ for all $x\in \O$,}
$$
which is \eqref{eq:grad-estim-gen}.\\

  To see \eqref{eq:normal-Hold-intro}, we argue as in the proof of the case $2s>1$ in Section \ref{ss:Th1}. Indeed, we start  by noting that,      using  \eqref{eq:estimU} and \eqref{eq:after-iterr},  we can find a constant  $c_0=c_0(\O,N,s)>0$ such that 
$$
\|\ov v_x\|_{C^{1,\s-1+2s}( B(x,c_0\d(x)))}\leq C \d^{s+\min(s,\a )}(x),
$$
provided $\a\not=s$.
This implies that
\be \label{eq:est--1}
 \|\n u -\ov \psi(x)\n U\|_{L^{\infty}(B(x,c_0\d(x)))}\leq C \d^{s+\min(s,\a )-1}(x)
\ee
and   for all $\b\in (0,\s-1+2s]$
\be\label{eq:est--2}
 [\n u -\ov \psi(x)\n U]_{C^{\b}(B(x,c_0\d(x)))} C \d^{s+\min(s,\a )-1-\b}(x). 
\ee
To finish the proof, we proceed as in Section \ref{ss:Th1}. Define $\ov w_x=\delta^{1-s}\left(\n u -\ov \psi(x)\n U\right)$ and then, for $y_1,y_2\in B(x,c_0\d(x))$ and $\b \in(0,\min(\s -1+2s)]$ we have
\begin{align}
|\ov w_x(y_1)-& \ov w_x(y_2)  |
\leq C \d^{ \min(s,\a )-\b}(x) |y_1-y_2|^{\b} \label{eq:est--3a},
\end{align} 
with the same calculation as in the proof of Theorem \ref{th:Gradient-estimate-2s-geq1} based on \eqref{eq:est--1},  \eqref{eq:est--2} and the regularity of $\d$. Therefore, provided $\a\not=s$, we have for $\b \in(0,\min(\a ,s,\s -1+2s)]$
\[
[\ov w_x]_{C^{\beta}(B(x,c_0\d (x)))}\leq C.
\]
Next, we define $V(y):=\d^{1-s}(y)\n u(y)$, and we note that 
$$
V(y)=\ov w_x(y)+ \d^{1-s}(y)\n U(y) \ov \psi(x).
$$
By assumption $\d^{1-s}\nabla U \in C^\g(\ov\O)$ and thus  $\d \n \Psi\in L^\infty(\O)$, so that by  \eqref{eq:grad-estim-gen} and \eqref{eq:Class-estim}
\be\label{eq:L-infty-bound}
\|V \|_{L^\infty(\O)}\leq C. 
\ee
On the other hand  by \eqref{eq:est--3a}    we have 
 $$
 [ V ]_{C^\b(  B(x,c_0\d(x)))}\leq   C \qquad\textrm{ for all $\b\in (0, \min\{\g,\a ,s , \s-1+2s \}],$  }
 $$
 provided $\a\not=s$.  
Arguing similarly as in the proof of \cite[Proposition 1.1]{{RS-2}} we have $[V]_{C^\b(\ov\O)}\leq C$. This together with \eqref{eq:L-infty-bound} yield  $(ii)$.
\end{proof}

\begin{proof}[Proof of Corollary \ref{cor:sleq12} ]
First, we observe that the results hold in $\O_\b:=\{x\in \O\,:\, \d(x)\geq \b\}$ for all $\b>0$.  From now on, we fix $\b>0$ small such that any point $x\in \O\setminus \O_\b$ has a unique projection $\s(x)\in \de\O$ and that the map $x\mapsto \s(x)$ is $C^1 (\ov \O\setminus \O_\b)$.\\
By \cite{AR19}, there exists a function $U\in C^{s}(\R^N)\cap C^2_{loc}(\O)$ satisfying \eqref{eq:torsion}, \eqref{eq:estimU}  and, for all $\e\in (0,1)$, there exists a constant $C_0=C_0(s,\O,N,\e)$ such that  
\be \label{eq:prop-of-reg-dist}
 U^{\frac{1}{s}} \in C^{1,1-\e}(\ov\O) \qquad\textrm{ and }\qquad  |D^2 U^{\frac{1}{s}} (x)|\leq C_0 \d^{-\e}(x) \qquad\textrm{ for all $x\in\O $.} 
\ee
From now on, we fix $\e=1-\a\in (0,1)$. In the following the constant $C$ is as in the statement of the corollary.\\
Since $U=0$ on $\de\O$, for $x\in \O\setminus \O_\b$,  by the mean value theorem, we have 
\be \label{eq:expand-U1s}
 U^{\frac{1}{s}}(x)=\d(x)\int_0^1\n   U^{\frac{1}{s}}(x-t \d(x) \n\d(\s(x)))\cdot \n\d(\s(x))\, dt .
\ee
From this, \eqref{eq:prop-of-reg-dist}  and the fact that $\d$ is smooth on $\de\O$ (in fact it vanishes there),  we deduce that for   every $x\in \O\setminus \O_\b$,  
$$
|\n (U^{\frac{1}{s}}  /\d )|\leq  C_0  \int_0^1  \d\left(x-t \d(x)  \n\d(\s(x))\right)^{\a-1} \,dt+ C\leq C \d^{\a-1}(x).
$$
It then follows that, for   every $x\in \O\setminus \O_\b$,  
\begin{align*}
|\n \Psi(x)|=|\n (U^{\frac{1}{s}}  /\d)^s(x)|&=s (U^{\frac{1}{s}}  /\d)^{s-1} (x)|\n(  U^{\frac{1}{s}}  /\d)(x)| \leq C  \d^{\a-1}(x),
\end{align*}
where we used \eqref{eq:estimU}. Hence by  the regularity of $U$ and $\d$, we obtain
$$
|\n \Psi(x)| \leq C  \d^{\a-1}(x)\qquad\textrm{ for all $x\in \O$.}
$$
  Therefore by applying   Theorem \ref{th:Gradient-estimate-gen} $(i)$, we obtain $(i)$.\\

We now prove $(ii)$.    We start by noting that,  by \eqref{eq:prop-of-reg-dist},  \eqref{eq:expand-U1s} and \eqref{eq:estimU} we get 
\be \label{eq:classical-norm-estim}
\frac{ U^{\frac{1}{s}}}{  \d}\in C^{\a}(\ov \O\setminus \O_\b) \qquad\textrm{ and } \qquad  \frac{ U^{\frac{1}{s}}}{  \d} \geq C \qquad\textrm{ in  $\ov \O\setminus \O_\b$.}
\ee
 Direct computations yield
$$
\d^{1-s}\n U= \d^{1-s} \n ( U^{\frac{1}{s}})^s=s  \d^{1-s}( U^{\frac{1}{s}})^{s-1} \n  U^{\frac{1}{s}}=s (U^{\frac{1}{s}}  /\d)^{s-1}  \n  U^{\frac{1}{s}}.
$$
From this, \eqref{eq:classical-norm-estim}, \eqref{eq:prop-of-reg-dist} and   \eqref{eq:expand-U1s} we obtain  $\d^{1-s}\n U\in C^{\a}(\ov \O\setminus \O_\b)$.  We then conclude, from the regularity of $U$ and $\d$, that $\d^{1-s}\n U\in C^{\a}(\ov \O)$. Now by Theorem \ref{th:Gradient-estimate-gen} $(ii)$ we have $\|\d^{1-s}\n u\|_{C^\b(\ov\O)}\leq C$ and the proof of \eqref{eq:normal-Hold-intro} is complete. 

Finally \eqref{eq:norm-deriv-sleq12} follows from  \eqref{eq:Class-estim},   \eqref{eq:estim-Grad-ok} and  \eqref{eq:normal-Hold-intro}. 

\end{proof}


\begin{thebibliography}{99}

\bibitem{AR19} N.~Abatangelo, X. Ros-Oton, Obstacle problems for integro-differential operators: higher regularity of free boundaries, preprint (2019).

\bibitem{Fall-reg-1} M.M. Fall, Regularity estimates for nonlocal Schr\"odinger equations.  Discrete and Continuous Dynamical Systems - A, 2019, 39 (3) : 1405-1456.

\bibitem{FJ15}
M.~M.~Fall and S.~Jarohs, Overdetermined problems with fractional Laplacian, ESAIM Control Optim. Calc. Var. \textbf{21.4} (2015), 924--938.

\bibitem{Grubb1} G. Grubb, Fractional Laplacians on domains, a development of Hormander's theory of $\mu$-transmission pseudodifferential operators, Adv. Math. 268 
(2015), 478-528.

\bibitem{RS16a}  X. Ros-Oton, J. Serra, Regularity theory for general stable operators, J. Differential Equations 260 (2016), 8675-8715.

\bibitem{RS16b} 
X. Ros-Oton, J. Serra,  Boundary regularity for fully nonlinear integro-differential equations, Duke Math. J. 165 (2016), no. 11, 2079--2154.

\bibitem{RS-1}  X. Ros-Oton, J. Serra, Boundary regularity estimates for nonlocal elliptic equations in $C^1$ and $C^{1,\a}$ domains. Ann.
Mat. Pura Appl. (4) 196 (2017), no. 5, 1637-1668.

\bibitem{RS-2} X. Ros-Oton, J. Serra, The Dirichlet problem for the fractional Laplacian: regularity up to the boundary, J. Math.
Pures Appl. (9) 101 (2014), no. 3, 275-302.

\end{thebibliography}
\end{document}